\documentclass[a4paper,12pt]{amsart}

\usepackage[centertags]{amsmath}
\usepackage{amsfonts}
\usepackage{amscd}
\usepackage{mathtools}
\usepackage[upright]{fourier}
\usepackage{kerkis}
\usepackage{eucal}
\usepackage{parskip}
\newcommand{\B}[1]{\mathbf {#1}}
\newcommand{\C}[1]{\mathcal {#1}}
\newcommand{\F}[1]{\mathfrak {#1}}

\newtheorem{theorem}[equation]{Theorem}
\newtheorem{corollary}[equation]{Corollary}
\newtheorem{lemma}[equation]{Lemma}
\newtheorem{proposition}[equation]{Proposition}
\theoremstyle{definition}
\newtheorem{example}[equation]{Example}
\theoremstyle{remark}

\newtheorem{remark}[equation]{Remark}
\newtheorem{question}[equation]{Question}
\numberwithin{equation}{section}
\numberwithin{figure}{section}
\numberwithin{table}{section}

\newcommand\OP{\operatorname}
\newcommand\Ham{\OP{Ham}}
\newcommand\Diff{\OP{Diff}}

\begin{document}

\title{On bi-invariant word metrics}
\author{\'Swiatos\l aw R. Gal}
\address{\'SRG --- Uniwersytet  Wroc\l awski \& Universit\"at Wien}
\email{sgal@math.uni.wroc.pl}
\author{Jarek K\k edra}
\address{JK --- University of Aberdeen \& Uniwersytet Szczeci\'snski}
\email{kedra@abdn.ac.uk}

\begin{abstract}
We prove that bi-invariant word metrics are bounded
on certain Chevalley groups. As an application we provide
restrictions on Hamiltonian actions of such groups.
\end{abstract}

\maketitle


\section{Introduction}\label{S:intro}

\subsection{The result}\label{SS:result}
Let $\C O_{\sf V}\subset \F K$ be a ring of $\sf V$-integers in a number
field~$\F {K}$, where $\sf V$ is a set of valuations containing all 
Archimedean ones. Let $G_{\pi}(\Phi,\C O_{\sf V})$ be the Chevalley
group associated with a faithful representation
$\pi\colon \F g\to \F{gl}(V)$ of a simple
complex Lie algebra of rank at least two
(see Section \ref{S:chevalley} for details).

\begin{theorem}\label{T:main}
Let\/ $\Gamma$ be a finite extension or a supergroup of finite index of
the Chevalley group $G_{\pi}(\Phi,\C O_{\sf V})$. Then any bi-invariant
metric on $\Gamma$ is bounded.
\end{theorem}

The examples to which Theorem \ref{T:main} applies include the
following groups. 
\begin{enumerate}
\item
$\OP{SL}(n;\B Z)$; it is a non-uniform lattice in $\OP{SL}(n;\B R)$.
\item
$\OP{SL}\left(n;\B Z\left[\sqrt{2}\right]\right)$; the image of its diagonal
embedding into the product
$\OP{SL}(n;\B R)\times \OP{SL}(n;\B R)$
is a  non-uniform lattice.
\item
$\OP{SO}(n;\B Z) := \left\{A\in \OP{SL}(n,\B Z)\,|\, AJA^T=J\right\},$
where $J$ is the matrix with ones on the anti-diagonal and zeros
elsewhere. It is a non-uniform lattice in the split real
form of $\OP{SO}(n,\B C)$.
\item
$\OP{Sp}(2n;\B Z)$; it is a non-uniform lattice in
$\OP{Sp}(2n;\B R)$. Its nontrivial central extension
by the infinite cyclic group has unbounded bi-invariant
word metric (see Example \ref{E:central}).
\end{enumerate}
The details of these and other examples are presented in 
Section \ref{SS:properties} on page \pageref{SS:properties}.

\subsection{Remarks}
The proof of Theorem \ref{T:main} follows from the boundedness
of the bi-invariant word metric. This is the usual word metric
induced by a set of generators invariant under conjugation.
Such generating sets are in general infinite. However, if
a group $\Gamma$ is generated by conjugates of finitely many
elements then the Lipschitz equivalence class of such metrics
is well defined and it is maximal among all bi-invariant metrics.
In particular, if such a bi-invariant word metric is bounded
then so is any bi-invariant metric.

Thus the proof amounts to showing that the bi-invariant word metric
is bounded on a group $\Gamma$ as in Theorem \ref{T:main}.  It
is a combination of two known facts.  The first is that the group
$G(\C O_{\sf V})$ has bounded generation. This means that there is a
subset $X\subset G(\C O_{\sf V})$ and a number $m\in \B N$ such that
every element $g\in G(\C O_{\sf V})$ is a product of at most $m$
elements from $X$.  The second fact is that the bi-invariant word norm
is bounded on~$X$. The details are presented in Section~\ref{S:chevalley}.

Although the bounded generation is inherited by finite index
subgroups, bi-invariant word metrics do not behave well with this
respect.  For example, $\B Z$ is an index two unbounded subgroup in
the infinite dihedral group that is bounded (Example \ref{E:dihedral}).
In Section \ref{S:below}, we present (mostly well known) tools
used to prove unboundedness of bi-invariant word metrics.

\begin{question}
Suppose that $G$ is semisimple real Lie groups of higher rank and with
finite centre. Is a lattice $\Gamma \subset G$ bounded with respect to
the bi-invariant word metric?
\end{question}

Notice that certain lattices in groups of rank 1 admit 
nontrivial homogeneous quasi-homomorphisms which implies that their
bi-invariant word metrics are unbounded (see Lemma \ref{L:qm_S}).

The commutator length (on a perfect group), and its
generalisation due to Calegari and Zhuang \cite{calegari_zhuang} called
the $W$-length, as well as the torsion length (on a group generated by
torsion elements) \cite{MR2073290} induce bi-invariant
metrics. Although they are all interesting in their own rights our
main motivation for understanding bi-invariant word metrics was
different.

\subsection{A motivation and application}
Let $(M,\omega)$ be a symplectic manifold and let $\Ham(M,\omega)$ denote the group
of {\em compactly supported} Hamiltonian diffeomorphisms of $(M,\omega)$.  It
admits a bi-invariant metric, called the Hofer metric (see Section
\ref{SS:hofer} for definition). This metric is known to have infinite
diameter in many cases and no example of a symplectic manifold of
positive dimension with the Hofer metric of finite diameter is known.  We
would like to understand the algebraic structure of $\Ham(M,\omega)$ in the
sense of the following question.
\begin{question}
What are the finitely generated subgroups of the group
of Hamiltonian diffeomorphisms of a symplectic manifold?
\end{question}

\begin{corollary}\label{C:main}
Let\/ $(M,\omega)$ be a symplectic manifold. Let\/ $\Gamma$ be a finite
extension or a supergroup of finite index of a
Chevalley group $G_{\pi}(\Phi,\C O_{\sf V})$.
Then the image of a homomorphism
$$
\Gamma \to \Ham(M,\omega)
$$ 
lies within a bounded distance from the identity with respect to
the Hofer metric.
\end{corollary}

\begin{proof}
A homomorphism $\varphi\colon \Gamma \to \Ham(M,\omega)$ is Lipschitz with
respect to the bi-invariant word metric on $\Gamma$ and the Hofer
metric on the group of Hamiltonian diffeomorphism (see Lemma
\ref{L:lipschitz}). Since the bi-invariant word metric is bounded on
$\Gamma$ according to Theorem~\ref{T:main}, the image of $\varphi$ is
bounded in $\Ham(M,\omega)$.
\end{proof}

\subsection{Remarks}\label{SS:remarks}
There are examples of nontrivial Hamiltonian actions of arithmetic
lattices on closed symplectic manifolds (see Example
\ref{E:nontrivial_action}). In all examples known to us, such an
action factors through a compact group.

The well known result of Polterovich \cite{MR2003i:53126} states that
there are no nontrivial Hamiltonian actions of certain lattices on
symplectically hyperbolic manifolds \cite{MR2547825}.


\begin{question}
What are the bounded (and finitely generated) subgroups of $\Ham(M,\omega)$?
\end{question}

\begin{question}
Is there a closed symplectic manifold $(M,\omega)$ and a lattice $\Gamma$ in
a semisimple Lie group such that there exists a homomorphism $\Gamma
\to \Ham(M,\omega)$ which does not factor through a compact group?
\end{question}

\section{Preliminaries on bi-invariant word metrics}
\label{S:basic}
\subsection{The word metric}\label{SS:word}
Let $\Gamma$ be a group generated by a set $S\subset \Gamma$.
The {\bf word norm} on $\Gamma$ with respect to $S$ is defined
by
$$
|g|_S = 
\min\{k\in \B N\,|\,g=s_{i_1}\ldots s_{i_k}, \text{ where } s_i\in S\}.
$$
Suppose that $S=S^{-1}$.
It is a standard fact that the above function satisfies the
following properties for all elements $g,h\in \Gamma$.
\begin{enumerate}
\item
$|g|_S \geq 0$
\item
$|g|_S = 0$ if and only if $g=\OP{1}$
\item
$|gh|_S\leq |g|_S+|h|_S$
\item
$|g^{-1}|_S=|g|_S$
\end{enumerate}
Such a norm defines a right-invariant metric $d_S$ on $\Gamma$ by
$$
d_S(g,h)=|gh^{-1}|_S.
$$ 
The metric is called the {\bf word metric} associated with the
generating set $S$. The geometry of such metrics for finitely
generated groups has been a subject of extensive research during the
last few decades originating in Gromov \cite{MR1253544}.

\subsection{Bi-invariant word metrics}\label{SS:biword}
If the generating set $S$ is invariant under the conjugation then so
is the norm.  That is, for all $g,h\in \Gamma$ we have
$$
|hgh^{-1}|_S=|g|_S.
$$
The induced metric is then bi-invariant.  Let $S\subset \Gamma$ be a
subset normally generating $\Gamma$. This means that $\Gamma$ is
generated by
$$
\overline S:=\bigcup_{g\in \Gamma} gSg^{-1}.
$$
The set $\overline S$ is invariant under the conjugation. 
If $S$ is a generating set then, since
$S\subset \overline S$, we have
$$
|g|_{\overline S}\leq |g|_S
$$
for every $g\in \Gamma$. If $\Gamma$ is Abelian then the two
norms coincide. In general, the bi-invariant norm is strictly
smaller on some elements. 

\begin{remark}
The basic properties presented in this section are elementary and can
be found in the paper of Burago, Ivanov and Polterovich
\cite{MR2509711}, where they investigate bi-invariant metrics on
groups of diffeomorphisms of manifolds.

In general, not much is known for bi-invariant word metrics with an
exception for the commutator length (see Section \ref{SS:scl}). 
\end{remark}

\begin{example}\label{E:f2}
Let $\Gamma=\OP{F}_2$ be a free group generated by two elements.
Let $S=\{x,x^{-1},y,y^{-1}\}$. Observe that
$$
|y^nxy^{-n}|_{\overline S }=1 \text{ and } |y^nxy^{-n}|_S=2n+1.
$$
\hfill $\diamondsuit$
\end{example}

\begin{example}\label{E:conjugacy}
If $\Gamma$ has finitely many conjugacy classes then
any bi-invariant metric is bounded. \hfill $\diamondsuit$
\end{example}

\subsection{The Lipschitz property}\label{SS:lipschitz}

A group $G$ is {\bf normally finitely generated} if there exists a
finite set $S\subset G$, such that $G$ is generated by all the
conjugates of elements of $S$.

\begin{lemma}\label{L:lipschitz}
Let\/ $\Gamma$ be a group normally generated by a finite set\/ $S=S^{-1}$.
Let\/ $G$ be a group equipped with a bi-invariant norm $\|\phantom{g}\|$.
A homomorphism $\psi\colon \Gamma \to G$ is Lipschitz. That is,
there exists a positive constant\/ $\mu\in \B R$ such that
$$
\|\psi(g)\|\leq\mu|g|_{\overline S}
$$
for every $g\in \Gamma$.
\end{lemma}

\begin{proof}
Let $\mu:=\max\{\|\psi(s)\|\,|\,s\in S\}$ and let
$g=s_{i_1}^{h_1}\ldots s_{i_k}^{h_k}$ be a word of minimal
length showing that  $|g|_{\overline S}=k$. Then
\begin{eqnarray*}
\|\psi(g)\|&=& \left\|\psi(s_{i_1})^{\psi(h_1)}\ldots \psi(s_{i_k})^{\psi(h_k)}\right\|\\
&\leq& \sum_{j=1}^k\|\psi(s_{i_j})\|\\
&\leq& \mu k = \mu|g|_{\overline S}
\end{eqnarray*}
\end{proof}

\begin{example}\label{E:simple}
If $G$ is a simple group then it is normally generated
by $\{g,g^{-1}\}$ for any $g\neq \OP{Id}$. Let us apply this
to the group of Hamiltonian diffeomorphisms of a
closed symplectic manifold. Let $\OP{Id}\neq g\in \Ham(M,\omega)$
and let $S=\{g,g^{-1}\}$.
Then
$$
|f|_{\overline S} \geq \frac{1}{\|g\|}\|f\|
$$
for every $f\in \Ham(M,\omega)$ where $\|f\|$ denotes the Hofer norm.  
\hfill $\diamondsuit$
\end{example}

The Lipschitz equivalence class of a bi-invariant word metric on a
normally finitely generated group is well defined.  And this class is
maximal in the sense that any other bi-invariant metric is Lipschitz
with respect to it. More precisely, the identity from the word metric
to any other bi-invariant metric is Lipschitz.  In particular, a
normally finitely generated group $G$ admits an unbounded bi-invariant
metric if and only if the bi-invariant word metric is unbounded.

\subsection{Convention}\label{SS:convention}
In what follows, we shall frequently abuse terminology
by saying {\bf the bi-invariant word metric} having in mind
the Lipschitz equivalence class of such metrics. The notation
$d_{\Gamma}$ will mean $d_{\overline S}$ for some finite
generating set $S\subset \Gamma$.

\begin{corollary}\label{C:h1}
Let\/ $\Gamma$ be a group normally generated by a finite set.  If the
associated bi-invariant word metric is bounded then every quotient
of $\Gamma$ is bounded. In particular,
the abelianisation $\Gamma/[\Gamma,\Gamma]$ is finite. \qed
\end{corollary}

\begin{example}\label{E:dihedral}
Let $\Gamma = \B Z/2 \star \B Z/2$ be the infinite dihedral group. It
is direct calculation that the bi-invariant word metric is bounded by
$2$.  On the other hand, $\Gamma$ contains an infinite cyclic subgroup
of index two which has, of course, unbounded bi-invariant metric. This
shows that the inclusion of a finite index (normal) subgroup is not
Lipschitz.  \hfill $\diamondsuit$
\end{example}

\begin{corollary}\label{C:section}
Let\/ $\Gamma $ be a finitely generated group and
let\/ $\pi\colon \Gamma \to \Delta$ be a surjective homomorphism.
If $s\colon \Delta \to \Gamma$ is a left inverse of $\pi$, that
is, $s\circ\pi = \OP{Id}_{\Delta}$, then 
the bi-invariant word metric on $\Delta$ induced from $\Gamma$
is equivalent to the bi-invariant word metric $d_{\Delta}$.
\qed
\end{corollary}

\begin{corollary}\label{C:quotient}
If\/ $F\to \Gamma \to \Delta$ is a split extension with finite
kernel then the quotient map is a bi-Lipschitz equivalence.
\qed
\end{corollary}

\section{When a bi-invariant metric is unbounded?}
\label{S:below}
Most of the material presented in this section is known and
standard, except possibly for the part about extensions.

\subsection{Distortion \cite[Section 3]{MR1253544}}
\label{SS:distortion}\hfill

Let $G$ be a group equipped with a norm $\|\phantom{g}\|$.
The {\bf translation length} of an element $g\in G$
is defined by
$$
\tau(g):=\lim_{n\to \infty}\frac{\|g^n\|}{n}.
$$

\begin{lemma}\label{L:tau}
For every $g\in G$ we have
$\tau(g)=\inf_{n}\frac{\|g^n\|}{n}$.
\end{lemma}

\begin{proof}
First, observe that $0\leq \frac{\|g^n\|}{n}\leq \|g\|$ and hence the
above infimum exists. Let us denote this infimum by $\mu$. Let
$\epsilon>0$ and let $m\in \B N$ be such that
$$
\frac{\|g^m\|}{m}< \mu+\epsilon.
$$
Choose an $n>\frac{m\|g\|}{\epsilon}$ and write it as $n=am+b$, for
some $a,b\in \B N$ with $b<m$. Then we have
$$
\|g^n\| =\|g^{am+b}\|\leq \|g^{am}\|+\|g^b\|\leq a\|g^m\|+\|g^b\|.
$$
It follows that
$$
\frac{\|g^n\|}{n}\leq \frac{\|g^m\|}{m}+\epsilon < \mu+2\epsilon.
$$
Since $\epsilon$ is arbitrary, this shows that
$$
\lim_{n\to\infty}\frac{\|g^n\|}{n}\leq  \mu
$$
which finishes the proof.
\end{proof}

An element $g\in G$ is called {\bf distorted} with respect
to the norm $\|\phantom{g}\|$ if its translation length is
equal to zero and {\bf undistorted} otherwise.

\begin{lemma}
Let\/ $\psi\colon \Gamma\to G$ be a homomorphism. If $\psi(g)$ is
undistorted then so is $g$.
\qed
\end{lemma}

\subsection{The commutator length \cite{MR2527432}}\label{SS:scl}
Let $G$ be a group. The commutator length
$$
\OP{cl}\colon [G,G]\to \B R
$$ 
is defined to be the length of the shortest word expressing $g$ and
consisting of commutators of elements from $G$. Notice that if
$G=[G,G]$ then the commutator length is a bi-invariant norm on $G$.
Thus following observation is direct consequence of
Lemma \ref{L:lipschitz}.

\begin{lemma}\label{L:cl_S}
Let\/ $\Gamma$ be a perfect group generated by a finite set\/ $S=S^{-1}$.
Then there exists a constant\/ $\nu>0$ such that
$$
\OP{cl}(g)\leq \nu|g|_{\overline S}
$$
for every $g\in \Gamma$. In other words, the bi-invariant word norm
is Lipschitz with respect to the commutator length. \qed
\end{lemma}

The {\bf stable commutator length} of an element $g\in G$ is
defined as the translation length with respect to the commutator
length. That is,
$$
\OP{scl}(g):=\lim_{n\to \infty}\frac{\OP{cl}(g^n)}{n}.
$$
The previous lemma has an immediate corollary.
\begin{corollary}\label{C:scl}
Let $\tau_{\overline S}$ denote the translation length of
the $\overline S$-word norm on a perfect group $\Gamma$ generated
by a finite set $S$. Then
$$
\OP{scl}(g)\leq \nu\tau_{\overline S}(g)
$$
for every $g\in \Gamma$ and the constant $\nu$ from Lemma \ref{L:cl_S}.
\qed
\end{corollary}

\subsection{Quasi-homomorphisms \cite{MR2026941}}\label{SS:qm}

Let $G$ be a group. A {\bf quasi-ho\-mo\-mor\-phism} 
$$
q\colon G\to \B R
$$ 
is a function such that there exist a constant $D\geq 0$ (called the
{\bf defect of $q$}) such that
$$
|q(g)-q(gh)+q(h)|\leq D
$$
for every $g,h\in G$. A quasi-homomorphism is called {\bf homogeneous}
if 
$$
q(g^n) = nq(g)
$$
for every $g\in G$ and every $n\in \B Z$. 
If $q\colon G\to \B R$ is a quasi-homomorphism with defect
$D$ then the formula
$$
\widehat q(g):=\lim_{n\to \infty}\frac{q(g^n)}{n}
$$
defines a homogeneous quasi-homomorphism and we have
$$
|q(g)-\widehat q(g)|\leq D
$$ 
for all $g\in G$. Thus if $q$ is unbounded then so is its homogenisation
and if $q$ is bounded then its homogenisation is identically zero.

\begin{lemma}{\cite[Theorem 2.70]{MR2527432}}\label{L:qm}
Let\/ $G$ be a perfect group and let\/ $q\colon G\to \B R$ be
a homogeneous quasi-homomorphism. Then there exists a positive
constant\/ $C>0$ such that
$$
|q(g)|\leq C\OP{scl}(g)
$$
for every $g\in G$.
\end{lemma}

\begin{lemma}\label{L:qm_S}
Let\/ $q\colon \Gamma\to \B R$ be a quasi-homomorphism
of a group generated by a finite set\/ $S$.
Then there exists a constant\/ $C>0$ such that
$$
|q(g)|\leq C|g|_{\overline S}
$$
for every $g\in \Gamma$.
\end{lemma}

\begin{proof}
Let $\mu:=\max\{|q(s)|\,|\,s\in S\}$.
Let $g\in \Gamma$ be of $\overline S$-length equal to $k$.
That is, $g=s_{i_1}^{h_1}\ldots s_{i_k}^{h_k}$. The following
calculation follows directly from the quasi-homomorphism
property of $q$. 
\begin{eqnarray*}
|q(g)|&=&\left|q\left(s_{i_1}^{h_1}\ldots s_{i_k}^{h_k}\right)\right|\\
&\leq & (k-1)D +\sum_{j=1}^k\left|q\left(s_{i_j}^{h_j}\right)\right|\\
&\leq & (k-1)D +\sum_{j=1}^k(2D+|q(s_{i_j})|)\\
&\leq & kD+2kD+k\mu\\
&\leq & (3D +\mu)|g|_{\overline S}
\end{eqnarray*}
\end{proof}

\begin{corollary}\label{C:qm_S}
If $q\colon \Gamma\to \B R$ is a homogeneous quasi-homomorphism
on a group generated by a finite set\/ $S$ then there exists a constant
$C>0$ such that
$$
|q(g)|\leq C\tau_{\overline S}(g)
$$
for every $g\in \Gamma$. Consequently, if $q(g)\neq 0$ then
$g$ is undistorted with respect to the bi-invariant word metric.
\qed
\end{corollary}

\begin{example}
If $\Gamma$ is a hyperbolic group then, due to 
Epstein and Fujiwara \cite{MR1452851} the secound bounded cohomology
group is infinite dimentional, therefore the comparison map
$H^2_b(\Gamma;\B R)\to H^2(\Gamma;\B R)$ has kernel, and as a corollary
there exists nontrivial, thus unbounded quasimorphism on $\Gamma$
Consequently, the bi-invariant word
metric on a non-elementary hyperbolic group
is unbounded.
\hfill $\diamondsuit$
\end{example}

\subsection{Extensions}\label{SS:central}

\begin{proposition}\label{P:central_finite}
Let\/ $K\stackrel{i}\to \widehat \Gamma \stackrel{\pi}\to \Gamma$ be an
extension with bounded kernel~$K$.  Then $\widehat \Gamma$ is bounded
if and only if the group $\Gamma$ is bounded. In particular,
an extension of a bounded group by a finite group is bounded.
\end{proposition}

\begin{proof}
If the extension is bounded then the quotient is bounded
due to the Lipschitz property of the quotient homomorphism
(see Section \ref{SS:lipschitz}).

Suppose that the bi-invariant word metric of the quotient is bounded
by $m$.  Let\/ $s\colon \Gamma \to \widehat \Gamma$ be a section such
that $s(1)=1$. Let $\widehat S$ be a generating set of 
$\widehat \Gamma$ containing the image $s(S)$ of the generating set of
the quotient. 
Let $\kappa:=\max\{|i(k)|_{\widehat \Gamma}\,|\, k\in K\}$.

Let $\hat g\in \widehat \Gamma$ be any element. Let
$\pi(\hat g)=g_1\ldots g_m$.
The following calculation yields the proof.

\begin{eqnarray*}
|\hat g|_{\widehat \Gamma}&=& |s(\pi(\hat g))i(k)|_{\widehat \Gamma}\\
&=& |s(g_1\ldots g_m)i(k)|_{\widehat \Gamma}\\
&=& |s(g_1)\ldots s(g_m)i(k_1)\ldots i(k_m)i(k)|_{\widehat \Gamma}\\
&\leq & m+|i(k_1\ldots k_mk)|_{\widehat \Gamma}\\
&\leq & m+\kappa.
\end{eqnarray*}
\end{proof}

\begin{proposition}\label{P:central}
Let\/ $\B Z\stackrel{i}\to \widehat \Gamma \stackrel{\pi}\to \Gamma$ 
be a central extension associated with the class 
$0\neq [c]\in H^2(\Gamma,\B Z)$. If 
the cocycle $c$ is bounded then the 
image $i(\B Z)\subset \widehat \Gamma$ is unbounded. 
Consequently, $\widehat \Gamma$ is unbounded.
\end{proposition}

\begin{proof}
Let $B$ be a constant such that $|c(g,h)|\leq B$ for every 
$g,h\in \Gamma$.  Let $s\colon \Gamma \to \widehat \Gamma$ be a
section such that $s(1)=1$.  Let $S$ be a set normally generating
$\Gamma$. Then its image $s(S)$ normally generate $\widehat \Gamma$.
We consider the word metrics with respect to these sets.

Suppose, on the contrary to the statement, that
$i(\B Z)\subset \widehat \Gamma$ is bounded. That is
there exists a constant $C$ such that
$|i(k)|_{\Gamma}\leq C$ for every $k\in \B Z$.
For any $k\in \B Z$ we have the following equalities.
\begin{eqnarray*}
i(k) &=& s(g_1)s(g_2)\ldots s(g_m)\\
&=& s(g_1g_2)s(g_3)\ldots s(g_m)c(g_1,g_2)\\
&=& s(g_1g_2\ldots g_m)c(g_1,g_2)c(g_1g_2,g_3)\ldots c(g_1g_2\ldots g_{m-1},g_m)\\
&=& c(g_1,g_2)c(g_1g_2,g_3)\ldots c(g_1g_2,g_{m-1},g_m)\\
\end{eqnarray*}
It follows that $|k| \leq (m-1)B \leq CB$ which is a contradiction for
$k\in \B Z$ was chosen to be arbitrary.
\end{proof}

\begin{example}\label{E:central}
Let $\Gamma =\OP{Sp}(2n,\B Z)\subset \OP{Sp}(2n;\B R)$ be a lattice
and let $\widehat \Gamma$ be the central extension that is the pullback
of the universal cover $\widetilde {\OP{Sp}}(2n;\B R)\to \OP{Sp}(2n;\B R)$
with respect to the inclusion of the lattice.

It is known that this extension is associated with a bounded
cohomology class and hence $\widehat {\OP{Sp}}(2n;\B Z)$ is unbounded,
due to Proposition \ref{P:central}.  On the other hand, the quotient
$\OP{Sp}(2n;\B Z)$ is a Chevalley group and, according to Theorem
\ref{T:main}, it is bounded.
\hfill $\diamondsuit$
\end{example}

\section{Bi-invariant word metrics on Chevalley groups}\label{S:bounded}

\subsection{Chevalley groups}\label{S:chevalley}
Let $\pi\colon \F g\to \F {gl}(V)$ be a representation
of a complex semisimple Lie algebra and let $\C O$ be a commutative
ring with unit. Let $\Phi$ denotes the root system associated with
a Cartan subalgebra $\F h\subset \F g$.
With these data there are associated two groups
$$
\OP{E}_{\pi}(\Phi,\C O)\subset \OP{G}_{\pi}(\Phi,\C O)
$$ called the {\bf elementary Chevalley group} and the 
{\bf Chevalley group} respectively. If $\C O$ is a field then these
groups coincide and are well understood
\cite{MR0258840,MR0407163,MR0466335}.  The situation over rings is
much more delicate \cite{MR1409655}. Let us define the groups.

The elementary Chevalley group $\OP{E}_{\pi}(\Phi,\C O)$ is defined as
the subgroup of the automorphism group 
$\OP{Aut}(V_{\B Z}\otimes_{\B Z}\C O)$ generated by elements of the
form
$$
x_{\alpha}(t):=\exp(t \pi(x_{\alpha}))
$$ 
where $\alpha \in \Phi$ is a root and $t\in \C O$.
Here, $V_{\B Z}$ is an admissible $\B Z$-form of $V$,
that is, an integral lattice preserved by the
representation, see Borel \cite{MR0258840} for details.

Let $G\subset \OP{G}(n,\B C)$ be a complex Lie group corresponding to
the Lie algebra $\F g$, where the identification
$\OP{GL}(V)\cong \OP{GL}(n,\B C)$ is done via the basis of 
$V_{\B Z}$. This basis defines coordinate functions on 
$\OP{GL}(n,\B C)$ restrictions of which generate a Hopf subalgebra 
$\B Z[G]\subset \C [G]$ in the coordinate ring for $G$.  The Chevalley
group is defined to be an affine group scheme over the integers
$$
\OP{G}_{\pi}(\Phi,\C O):= \OP{Hom}(\B Z[G],\C O).
$$

\begin{remark}
Both definitions above depend on the choice of the admissible
$\B Z$-form $V_{\B Z}$. This choice is not mentioned in our
abused notation.
\end{remark}

It is not difficult to see that there is an inclusion
$\OP{E}_{\pi}(\Phi,\C O)\subset \OP{G}_{\pi}(\Phi,\C O)$.
Moreover, if $G$ is of rank at least two and 
$\C O =\C O_{\sf V}$ is the ring of $\sf V$-integers
in a number field then the two groups coincide
(see Tavgen' \cite[Lemma 4]{MR1044049}).

\begin{lemma}\cite[Theorem 5.2.2]{MR0407163}\label{L:chevalley}
The generators $x_{\alpha}(t)\in \OP{E}_{\pi}(\Phi,\C O)$ satisfy the
following commutation relations, 
$$
[x_{\alpha}(k),x_{\beta}(l)]
=\prod_{i,j}x_{i\alpha+j\beta}(C(-l)^ik^j)
$$
where $i,j$ are positive integers such that $i\alpha + j\beta$
is a root and $C$ is an integer such that $|C|\leq 3$. The
product is taken in the order of increasing $i+j$. \qed
\end{lemma}

\subsection{Proof of Theorem \ref{T:main}}
We shall prove the statement for the elementary Chevalley group.
Since it is of finite index in the Chevalley group the result
follows for the latter and for a general $\Gamma$ as in the
statement of theorem as well.

We shall show that there exists a positive number $m\in \B R$ such
that for every element $g\in \OP{E}_{\pi}(\Phi,\C O_{\sf V})$ and every
$n\in \B Z$ we have $|g^n|\leq~m$.  The first step is to prove this
claim for an element of the form 
$x_{\alpha}(r)\in \OP{E}_{\pi}(\Phi,\C O_{\sf V})$, where 
$\alpha \in \Phi$ is a root and $r\in \C O_{\sf V}$.

There exist a subsystem $\Psi\subset \Phi$ of rank two
isomorphic to either $\OP{\bf A}_2$ or $\OP{\bf B}_2$ and containing $\alpha$.
Indeed, there is a subsystem of rank two containing $\alpha$ and
it follows from the simplicity and the higher rank that
this system has to be simple of rank two, that is one
of $\OP{\bf A}_2,\OP{\bf B}_2,\OP{\bf G}_2$. It is then easy to see
that each root in $\OP{\bf G}_2$ is contained in some $\OP{\bf A}_2$.
 
Observe that there exists $\beta,\gamma \in \Psi$ such that 
$\alpha = \beta+\gamma$ and no other positive combination
of $\beta$ and $\gamma$ is a root. 
It follows from Lemma \ref{L:chevalley} that
$$
x_{\alpha}(-Ckr)=[x_{\beta}(kr),x_{\gamma}(1)]
$$
where $k\in \B Z$ and $C=\pm 1,\pm 2$ or $\pm 3$. 
Thus we obtain that
$$
|x_{\alpha}(Cr)^k| =|x_{\alpha}(r)^{Ck}| \leq 2
$$
which implies that the cyclic subgroup
generated by $x_{\alpha}(r)$ stays within a bounded
distance from the identity.

It is a result of Tavgen' \cite{MR1044049} that the group
$\OP{E}_{\pi}(\Phi,\C O_{\sf V})$ has bounded generation with respect
to the set of elements of the form $x_{\alpha}(r)$. That is, there is
a constant $B\in \B N$ such that every 
$g\in \OP{E}_{\pi}(\Phi,\C O_{\sf V})$ is a product of at most $B$
elements of the form $x_{\alpha}(r)$.

Let $\rho_1,\ldots,\rho_n\in \C O_{\sf V}$ be elements such that there
is an isomorphism 
$\C O_{\sf V}\cong \B Z\rho_1\oplus \ldots \oplus \B Z\rho_n$ 
of abelian groups. Let $S:=\{x_{\alpha}(\rho_i)\}$ be a set of
generators of $G_{\pi}(\Phi, \C O_{\sf V})$ and let $\mu$ be a number such that
$$
|x_{\alpha}(\rho_i)^k|_{\overline S}\leq \mu
$$ 
for all $1\leq i\leq n$, $\alpha \in \Phi$  and $k\in\B Z$.

Putting the two results together we obtain the following estimate
for every $g\in \OP{E}_{\pi}(\Phi,\C O_{\sf V})$.
\begin{eqnarray*}
|g|_{\overline S} 
&=& |x_{\alpha_1}(r_{1})^{k_1}\ldots x_{\alpha_l}(r_{l})^{k_l}|_{\overline S}\\
&=& |x_{\alpha_1}(\rho_{1})^{k_{11}}\ldots x_{\alpha_1}(\rho_{n})^{k_{1n}}\ldots
x_{\alpha_l}(\rho_{1})^{k_{l1}}\ldots x_{\alpha_l}(\rho_{n})^{k_{ln}}|_{\overline S}\\
&=& |x_{\alpha_1}(\rho_{1})^{k_{11}}|_{\overline S}+
\ldots +|x_{\alpha_l}(\rho_{n})^{k_{ln}}|_{\overline S}\\
&\leq & \mu ln \leq \mu Bn.
\end{eqnarray*}
\qed

\begin{corollary}\label{C:other}
Let $\Gamma$ be a group as in Theorem \ref{T:main}. Then
\begin{enumerate}
\item
if\/ $\Gamma$ is perfect then 
the commutator length is bounded on $\Gamma$ and
hence its stable commutator length is zero;
\item
if\/ $\Gamma$ is generated by torsion elements then
the torsion length is bounded;
\item
every quasi-homomorphism $q\colon \Gamma\to \B R$ 
is bounded.\qed
\end{enumerate}
\end{corollary}

\subsection{Properties and examples of Chevalley groups}\label{SS:properties}
Suppose that $G$ is an algebraic group defined over $\C O$.
It is clear from the definition of (elementary) Chevalley group
that we have inclusions
$$
\OP{E}_{\pi}(\Phi,\C O_{\sf V})\subset \OP{G}_{\pi}(\Phi,\C O_{\sf V})
\subset \OP{G}(\C O).
$$

In what follows we list various examples of bounded groups.
Notice that each example provides more groups by taking
finite extensions, quotients  and finite index supergroups.

\begin{example}\label{E:A_nZ}
The special linear group $\OP{SL}(n,\B Z)$ for $n\geq 3$ is bounded
because it is the Chevalley group $\OP{G}({\bf A}_{n-1},\B Z)$.  It is
a non-uniform lattice in $\OP{SL}(n,\B Z)$.
\hfill $\diamondsuit$
\end{example}

\begin{example}\label{E:A_nZ2}
The special linear group $\OP{SL}(n,\B Z[\sqrt{2}])$ for $n\geq 3$ is
bounded because it is the Chevalley group 
$\OP{G}({\bf A}_{n-1},\B Z[\sqrt{2}])$.  It is a non-uniform lattice in
$\OP{SL}(n,\B Z)\times \OP{SL}(n,\B Z)$.
\hfill $\diamondsuit$
\end{example}

\begin{example}\label{E:A_nZi}
The special linear group $\OP{SL}(n,\B Z[i])$ for $n\geq 3$ is bounded
because it is the Chevalley group $\OP{G}({\bf A}_{n-1},\B Z[i])$.  It is
a non-uniform lattice in $\OP{SL}(n,\B C)$.
\hfill $\diamondsuit$
\end{example}

\begin{example}\label{E:B_nZ}
Let $B$ be a quadratic form represented by the matrix with ones on the
antidiagonal and zeros elsewhere.  The associated orthogonal group
$\OP{SO}(n,n+1,\B Z)$ for $n\geq 2$ is bounded because it is the
Chevalley group $\OP{G}({\bf B}_{n},\B Z)$.  It is a non-uniform
lattice in $\OP{SO}(n,n+1)$. A similar example exists for the
root system $\OP{\bf D}_n$.
\hfill $\diamondsuit$
\end{example}

\begin{example}\label{E:C_nZ}
For the root system $\OP{\bf C}_n$ we obtain that 
the Chevalley group $\OP{G}({\bf C}_n,\B Z)$
is equal to $\OP{Sp}(2n;\B Z)$ and it is a non-uniform lattice
in the split real form $\OP{Sp}(2n;\B R)$.
\hfill $\diamondsuit$
\end{example}

\section{Hamiltonian representations}\label{S:hamiltonian}

\subsection{The Hofer metric \cite{MR96g:58001,MR2002g:53157}}\label{SS:hofer}
Let $(M,\omega)$ be a symplectic manifold. That is, $M$ is a smooth
manifold and $\omega$ is a closed and non-degenerate two-form.

Let $H\colon M\times \B R \to \B R$ be a smooth function.
It follows from he non-degeneracy of the symplectic form
that the time-dependent vector field satisfying the
identity
$$
\iota _{X_t}\omega = dH(\,\,,t)
$$
is well defined. Moreover, according to the closeness
of the symplectic form the flow $f_t$ of this vector
field preserves the symplectic form, $f^*_t\omega = \omega$.

A diffeomorphism obtained this way is called {\bf Hamiltonian} and the
group of Hamiltonian diffeomorphisms is denoted by $\Ham(M,\omega)$.

Let $f\in \Ham(M,\omega)$ be a compactly supported
Hamiltonian diffeomorphism generated by a function $H$.
The following formula
$$
\|f\|:= \inf_{H} \int_0^1 \OP{osc} H(\,\,,t) dt
$$
defines a bi-invariant norm called the {\bf Hofer norm}. The induced
bi-invariant metric is also called the {\bf Hofer metric}.
It is known to be unbounded in many cases
\cite{MR1979584,MR2017719}.

\subsection{Ad hoc examples of Hamiltonian representations}
\label{SS:ham_rep}

\begin{example}
Let $\Delta$ be a graph and let $\Gamma_{\Delta}$ be
the right-angled Artin group associated with $\Delta$.
That is, $\Gamma_{\Delta}$ is generated by the vertices
$v_1,\ldots,v_m$ of $\Delta$ modulo the following 
commutation relations $[v_i,v_j]=~\OP{1}$
if and only if $\{v_i,v_j\}$ is not an edge of $\Delta$.

Let $\{U_i\}\subset M$ be a family of open subset
such that their incidence graph is isomorphic to $\Delta$.
Let $f_i\in \Ham(M,\omega)$ be a diffeomorphism supported in
$U_i$. The map $\Gamma_{\Delta}\to \Ham(M,\omega)$ defined
by $v_i\mapsto f_i$ is clearly a homomorphism.

Such representations provide sometimes less obvious representations
due to the fact that right-angled Artin groups contain many
interesting subgroups \cite{MR2002k:20074}. The injectivity of certain
Hamiltonian actions of right-angled Artin groups on two-dimensional
sphere has been proven by M.Kapovich \cite{kapovich}.
\hfill $\diamondsuit$
\end{example}

\begin{example}
Suppose that $\Ham(M,\omega)$ contains two tori $T_1$ and $T_2$ with nonempty
intersection $T_1\cap T_2\neq \emptyset$. By choosing a finite number
of generators $f_1,\ldots,f_n\in T_1\cup T_2$ we obtain a
representation
$$
A\star _B C\to \Ham(M,\omega)
$$ 
of a finitely presented nontrivial amalgamated product of two
Abe\-lian groups.  For example, if $(M,\omega)$ is a Hirzebruch surface then,
in general, $\Ham(M,\omega)$ contains two dimensional tori satisfying the
above assumption \cite{MR1960129}. This way we obtain examples of
Hamiltonian actions of finitely presented groups which do not extend
to an action of a compact group.  \hfill $\diamondsuit$
\end{example}

Notice that in the above examples the groups are unbounded with
respect to the bi-invariant word metric (with few obvious exceptions).
Recall that Corollary \ref{C:main} states that the image of a
Hamiltonian action $\varphi\colon \Gamma\to \Ham(M,\omega)$, where $\Gamma$
is either a finite extension or a supergroup of finite index of a
Chevalley group is bounded.  The only examples of bounded subgroups of
the group of Hamiltonian diffeomorphisms of a closed symplectic
manifold we know are subgroups of compact Lie groups.

\begin{example}\label{E:nontrivial_action}
Let $H\subset \Ham(M,\omega)$ be the inclusion of a connected Lie group.
According to a theorem of Delzant \cite{delzant}, if $H$ is semisimple
then it is compact. Let $G$ be a non-compact connected semisimple Lie
group and let $\Gamma\subset G\times H$ be an irreducible
lattice. For example, let $J_{p.q}$ be the diagonal 
matrix with the first $p$ entries equal to one and the
last $q$ entries equal to $-\sqrt{2}$, where $p\geq q>1$.
Let 
$$
\Gamma:=\{A\in \OP{SL}(p+q,\B Z[\sqrt{2}])\,|\, AJA^T=J\}
$$
It is known that $\Gamma$ is a cocompact irreducible lattice in
$\OP{SO}(p,q)\times \OP{SO}(p+q)$.  Taking the composition of the
inclusion and the projection onto the second factor we obtain and a
highly nontrivial Hamiltonian action of $\Gamma$ on coadjoint orbits
of $\OP{SO}(2n)$. Notice, however that $\Gamma$ is not a Chevalley
group.
\hfill $\diamondsuit$
\end{example}

In the noncompact case the situation is different.

\begin{example}\label{E:sikorav}
Let $\B D(r)\subset \B R^{2n}$ be an open $2n$- dimensional disc
of radius $r$ centred at the origin.
It induces an inclusion
$$
\Ham(\B D(r),\omega_0)\subset \Ham(\B R^{2n},\omega_0).
$$
The Hofer diameter of $\Ham(\B D(r),\omega_0)$ is infinite since it
can be estimated from below by the absolute value of the
Calabi homomorphism.

On the other hand, the above inclusion is highly distorted.  It is a
result of Sikorav (Theorem 5 in Chapter 5 of Hofer-Zehnder
\cite{MR96g:58001}) that $\Ham(\B D(r),\omega_0)$ is within a bounded
distance from the identity with respect to the Hofer metric on
$\Ham(\B R^{2n})$.
\hfill $\diamondsuit$
\end{example}

\subsection{Other restrictions on actions of lattices}
\begin{example}\label{E:polterovich}
Let $\Gamma $ be a irreducible non-uniform lattice in a semisimple Lie
group of higher rank.  It is a result of Polterovich
\cite{MR2003i:53126} that if $(M,\omega)$ is a closed {\em symplectically
  hyperbolic} manifold then there is no non-trivial homomorphism
$\Gamma\to \Ham(M,\omega)$. It would be interesting to know if there
are nontrivial bounded subgroups in $\Ham(M,\omega)$ where $(M,\omega)$
is symplectically hyperbolic.
\hfill $\diamondsuit$
\end{example}

Our final comment is concerned with actions supported on a proper
subset of a manifold. In such cases there are severe restrictions
coming from Thurston Stability as presented by Franks in
\cite{MR2288284}. It implies that there is no chance for embedding a
lattice $\Gamma \subset G$ in a semisimple Lie group of higher rank
into a compactly supported diffeomorphisms of a non-compact manifold.
\begin{proposition}\label{P:thurston}
Let\/ $\Gamma $ be a finitely generated group.  
Let\/ $\Gamma \to \Diff(M)$ be a smooth effective action with support
strictly smaller than $M$. Then either $\Gamma $ is trivial or
it admits a nontrivial homomorphism $\Gamma \to \B R$.
\qed
\end{proposition}

\subsection*{Acknowledgements}
The authors would like to thank Dave Benson, Meinolf Geck, Tadeusz
Januszkiewicz, Colin Maclachlan, Nicolas Monod, Leonid Polterovich,
Geoff Robinson and Alain Valette for helpful discussions.
The authors thank Kamil Duszenko, Denis Osin, and Yehuda Shalom for pointing
out mistakes in an earlier version of the paper.

\'S.Gal is partially supported by Polish MNiSW grant N N201 541738, Swiss NSF
Sinergia Grant CRSI22-130435, and by the European Research Council (ERC) grant
of Goulnara Arzhantseva (grant agreement 259527)

\bibliography{../../bib/bibliography}
\bibliographystyle{acm}
\end{document}